\documentclass[a4paper,oneside, 11 pt, reqno]{amsart}
\usepackage{enumerate}
\usepackage{amsmath,amsthm,amscd,amsfonts,amssymb,mathrsfs} 
\usepackage{latexsym}
\usepackage[nobysame]{amsrefs}
\makeatletter
\@namedef{subjclassname@2020}{%
	\textup{2020} Mathematics Subject Classification}
\makeatother

\newcommand{\ncom}{\newcommand}
\ncom{\beqn}{\begin{eqnarray*}}
	\ncom{\eeqn}{\end{eqnarray*}}
\ncom{\beq}{\begin{eqnarray}}
	\ncom{\eeq}{\end{eqnarray}}
\ncom{\cal}{\mathcal}
\ncom{\eop}{\hfill{{\rule{2.5mm}{2.5mm}}}}
\ncom{\eoe}{\hfill{{\rule{1.5mm}{1.5mm}}}}
\ncom{\eof}{\hfill{{\rule{1.5mm}{1.5mm}}}}
\ncom{\hone}{\mbox{\hspace{1em}}}
\ncom{\htwo}{\mbox{\hspace{2em}}}
\ncom{\hthree}{\mbox{\hspace{3em}}}
\ncom{\hfour}{\mbox{\hspace{4em}}}
\ncom{\hsev}{\mbox{\hspace{7em}}}
\ncom{\vone}{\vskip 2ex}
\ncom{\vtwo}{\vskip 4ex}
\ncom{\vonee}{\vskip 1.5ex}
\ncom{\vthree}{\vskip 6ex}
\ncom{\vfour}{\vspace*{8ex}}
\ncom{\norm}{\|\;\;\|}
\ncom{\integ}[4]{\int_{#1}^{#2}\,{#3}\,d{#4}}
\ncom{\inp}[2]{\langle{#1},\,{#2} \rangle}
\ncom{\Inp}[2]{\Langle{#1},\,{#2} \Langle}
\ncom{\vspan}[1]{{{\rm\,span}\#1 \}}}
\ncom{\dm}[1]{\displaystyle {#1}}

\newtheorem{theorem}{\bf Theorem}[section]
\newtheorem{proposition}[theorem]{\bf Proposition}
\newtheorem{corollary}[theorem]{\bf Corollary}

\newtheoremstyle
{remarkstyle}
{}
{11pt}
{}
{}
{\bfseries}
{:}
{     }
{\thmname{#1} \thmnumber{#2} }

\theoremstyle{remarkstyle}

\newtheorem{example}[theorem]{\bf Example}

\def \D{\mathbb{D}}

\def \C{\mathbb{C}}

\def \T{\mathbb{T}}
\def \h{\mathbb{H}}
\def \H{\mathcal{H}}

\def \re{\mathrm{Re}}

\begin{document}

	\title[CH operators of finite rank \& de Branges-Rovnyak spaces]{Completely hyperexpansive operators with finite rank defect operator and de Branges-Rovnyak spaces}
	
	\author[S. A. Joshi]{Saee A. Joshi}
	\address{Ramnarain Ruia Autonomous College\\
		Mumbai- 400019, India}
	\email{saayalee@gmail.com}
	
	\author[V. M. Sholapurkar]{Vinayak. M. Sholapurkar}
	\address{ Bhaskaracharya Pratishthana\\
		Pune- 411004, India}
	\email{vmshola@gmail.com}
	
	\date{}
	
	\begin{abstract}
		The process of identifying a Dirichlet-type space $D(\mu)$ for a positive, Borel measure $\mu$, supported on the unit circle $\mathbb T,$  with a de Branges-Rovnyak space was initiated by Sarason \cite{sar}. A characterization of the symbol for a de Branges-Rovnyak spaces for which the shift operator is a $2$-isometry, was provided in \cite{kz}. In this paper, capitalizing on the Aleman's model \cite{ale} for the cyclic, analytic, completely hyperexpansive operators, we provide a characterization of  cyclic, analytic, completely hyperexpansive operator with finite rank defect operator in terms of the symbol for a  de Branges-Rovnyak space.  
	\end{abstract}
	\subjclass[2020]{Primary 47B20,  47B32; Secondary 47A67, 46E22}
	\keywords{de Branges-Rovnyak space, Dirichlet-type space, completely hyperexpansive operators}
	\maketitle
	\section{Introduction and Preliminaries}
	It is well known that the class of $2$-isometric operators  is a subclass of the class of  completely hyperexpansive operators. A theme of the present paper hinges around extending the results in \cite{kz} for $2$-isometric operators to completely hyperexpansive operators, especially in the case when the rank of the defect operators is finite. 
	
	In this paper, we describe a model for a cyclic, analytic, completely hyperexpansive operator $T$, if the defect operator  $\Delta_{T}= T^*T-I$ is of finite rank. A special case of such a model for $2$-isometric operators, has been obtained in \cite{sc}. This particular case has a bearing on the classical model  for a cyclic, analytic, $2$-isometric operator described by Richter (see \cite{rich}). Here, we capitalize on the classical model for a cyclic, analytic, completely hyperexpansive operator given by Aleman (see \cite{ale}) and impose an additional condition that the defect operator is of finite rank. In both the classical representations, the corresponding operator is shown to be unitarily equivalent to the multiplication operator $M_z$ on a Dirichlet space $D(\mu)$ where the measure is supported on the closed unit disk. In particular, if the operator is $2$-isometric, then the support of $\mu$ lies inside the unit circle.  In the circumstances where the defect operator is of finite rank, the de Branges-Rovnyak spaces make their appearance. Indeed, in a series of papers \cites{sc, cgr, kz, sar}, the authors have established interesting connections between Dirichlet spaces and de Branges-Rovnyak spaces. The most pertinent connections in the context of present paper are \cite[Theorem~1 and Corollary~3]{kz}. We record these results for a ready reference. \\
	\begin{theorem}{\label{t1.1}}
		Suppose that $b$ is non-extreme in the unit ball of $\h^\infty$. Then the restriction of the shift operator $S|_{\H(b)}$ is concave if and only if $$b(z)=\frac{c+\gamma z}{1-\beta z} , z\in\D \ \mbox{where} \ c,\gamma,\beta\in\C \ \mbox{with} \ |\beta|<1.$$ Furthermore, $S|_{\H(b)}$ is $2$-isometry if and only if $b(z)=\frac{c+\gamma z}{1-\beta z} , z\in\D$ where $c,\gamma,\beta\in\C$ are such that $1+|\beta|^2-|c|^2-|\gamma|^2=|\beta+\overline{c}\gamma|$. 
	\end{theorem}
	\begin{corollary}[Chevrot, Guillot, Ransford \cite{cgr}]\label{c1}
		Let $\mu$ be a finite positive measure on $\T$ and let $b$ be in the unit ball of $\h^\infty$. Then $D(\mu)=\H(b)$, with equality of norms, if and only if $$\mu=c\delta_\lambda \ \mbox{and} \ b(z)=\frac{\gamma z}{1-\beta z},$$ where $\beta\neq0,\ |\beta|+|\gamma|=1, \ c=\frac{|\gamma|^2}{|\beta|}$ and $\lambda=\frac{\overline{\beta}}{|\beta|}$.
	\end{corollary}
In this paper, we obtain an analogue of Theorem \ref{t1.1} for completely hyperexpansive operators. Further, we also obtain a version of Corollary \ref{c1} for the Dirichlet space $D(\mu)$, where the measure $\mu$ is supported on the closed unit disk. In the process, we also prove a generalization of \cite[Theorem~6.1]{sc}, extending the result to a completely hyperexpansive operator.  \\ 
The main result proved in this paper is stated below:  
 	\begin{theorem}\label{t1} Let $\mu$ be a finite, positive, Borel measure supported on $\overline{\D}$ and let $b\in \h^\infty$ be nonextreme with $\|b\|_\infty\leq 1$.
 	Then $D(\mu)=\H(b)$ with equality of norms, if and only if $\mu=|\alpha|^2\delta_\lambda  ,\ \mbox{for some } \alpha\in\C ,\ \lambda\in\overline{\D}$ and $b(z)=\frac{A_\lambda z}{1-B_\lambda z}$ where $ A_\lambda$ and $B_\lambda$ are constants depending on $\alpha$ and $\lambda$ such that 
 	$$|A_\lambda|^2=\begin{cases}
 		\frac{|\alpha|^2}{2|\lambda|^2}\left((1+|\alpha|^2+|\lambda|^2)-\sqrt{(1+|\alpha|^2+|\lambda|^2)^2-4|\lambda|^2}\right), & \mbox{ if} \ \lambda\neq 0\\
 		\displaystyle\frac{1}{1+|\alpha|^2}, & \mbox{ if} \ \lambda= 0
 	\end{cases}$$ and $$ B_\lambda= \frac{|A_\lambda|^2\overline{\lambda}}{|\alpha|^2}.$$	
 \end{theorem}
\noindent We obtain Corollary \ref{c1} as a special case of Theorem \ref{t1}.\\
\indent In the remainder of this section,  we set up the notation as well as the context necessary in the sequel. In section \ref{secA}, we prove that the defect operator of a multiplication operator $M_z$ on a Dirichlet space $D(\mu)$, where $\mu$ is supported on the closed unit disk, is of finite rank if and only if $\mu$ is finitely supported on the closed unit disk. In section \ref{secB}, we prove a generalization of Theorem \ref{t1.1} and finally, capitalizing on the results in sections \ref{secA} and \ref{secB}, we shall prove the main theorem in section  \ref{secC}. 
\\\indent For a complex Hilbert space $\cal H$, let $\cal B(\cal H)$ denote the $C^*$ algebra of bounded linear operators on $\cal H.$ An operator $T\in  \cal B(\cal H)$ is called \textit{cyclic} if there exists a vector $f\in \cal H$ such that $\vee\{T^nf:n\geq 0\}=\cal H.$  We say that $T$ is \textit{analytic} if $\cap_{n\geq 0}T^n\cal H=\{0\}.$  
	 An operator $T\in  \cal B(\cal H)$ is a \textit{$2$-isometry} if $$I-2T^*T+T^{*^2}T^2=0.$$ 
		An operator $T\in  \cal B(\cal H)$ is  a\textit{ $k$-hyperexpansive} if $$\sum_{i=0}^{n}(-1)^i{n \choose i} T^{*^i}T^{i}\leq 0 \ , \ \forall\ 1\leq n\leq k.$$
		An operator $T\in  \cal B(\cal H)$ is  \textit{completely hyperexpansive} if $$\sum_{i=0}^{n}(-1)^i{n \choose i} T^{*^i}T^{i}\leq 0 \ , \ \forall n\geq 1.$$

For an excellent account of the basic properties of $2$- isometries, the reader is referred to \cites{rich, rich1, rich2, rich3}.   The theory of completely hyperexpansive operators has been developed in \cites{ale, aa1, vms}.
	In what follows, we need to recall the definitions of two special types of Hilbert spaces of holomorphic function viz. the Dirichlet space and the de Branges-Rovnyak space. 		
	We shall denote the open unit disk in the complex plane by $\D$ and the unit circle by $\T.$ Let $\h^2$ denote the classical Hardy space of analytic functions on $\D$ having square integrable Taylor coefficients and $\h^\infty$ denote the space of all bounded analytic functions on $\D$.

	We shall be dealing with the Dirichlet type space $D(\mu)$ where the measure $\mu$ is supported on the closed unit disk $\overline{\D}.$ Aleman points out the delicacy in defining the local Dirichlet integrals for the points $\zeta\in \T$ (see \cite[Proposition~1.7, Chapter~IV]{ale}). Indeed, if  $\zeta \in \D,$ then the local Dirichlet integral of a function $f\in \H^2$ is defined as  
	$$ D_{\zeta}(f)=\left\|\frac{f-f(\zeta)}{z-\zeta }\right\|_{\h^2}^2 = \frac{1}{2\pi}\int_0^{2\pi} \left|\frac{f(e^{it})-f(\zeta)}{e^{it}-\zeta}\right|^2dt.$$ 
	If $\zeta \in \T,$ then the local Dirichlet integral is defined as 
	$$ D_{\zeta}(f)=\left\|\frac{f-f(\zeta)}{z-\zeta }\right\|_{\h^2}^2$$ if the nontangential limit $f(\zeta) $ of $f$ at $\zeta$ exists and $D_{\zeta}(f)=\infty $ otherwise. 
	
	As pointed out in \cite{ale}, the necessity of defining the local Dirichlet integrals for the boundary points, and thereby extending the support of  measure to the closed unit disk is forced by the operator theoretic considerations. Indeed, with this type of Dirichlet space, one can incorporate the case of multiplication operator $M_z$ being a $2$-isometry (see \cite[Theorem~1.10, Chapter~IV]{ale}).   Following \cite[Definition 1.8]{ale}, we recall the definition of  Dirichlet space under consideration for a ready reference. \\
		\indent	Let $\mu$ be a finite, non-negative measure supported on $\overline{\D}.$ \textit{The Dirichlet space $D(\mu)$}, consists of those functions $f\in \h^2$ with the property that $\int D_{\zeta}(f)d\mu(\zeta)< \infty .$   
	\noindent The Dirichlet space $D(\mu)$ is a reproducing kernel Hilbert space with $$ \|f\|^2=\|f\|_{\h^2}^2+\int D_{\zeta}(f)\mathrm{d}\mu(\zeta). $$ Further, as $\mu$ is finite, the multiplication operator $M_z$ is bounded. In fact, it turns out that $M_z$ on $D(\mu)$ is completely hyperexpansive and if $T\in \cal B(\cal H)$ is a completely hyperexpansive operator, then there exists a unique nonnegative measure $\mu$ supported on the closed unit disk $\overline{\D}$ such that $\cal H= D(\mu ) $ with the equality of norms and $T$ is unitarily equivalent to the operator $M_z$ on $D(\mu)$  (see \cite[Theorem 2.5, Chapter IV]{ale}). Further, $M_z$ is a $2$-isometry if and only if $\mu$ is supported on $\T$ (see \cite[Theorem 1.10(ii), Chapter IV ]{ale}). 
	
	We now turn our attention to the definition of a de Branges Rovnyak space. 	
	
	\indent Following \cite{vol2}, recall that 
			\textit{the de Branges-Rovnyak space $\H(b)$} associated with a symbol $b\in\h^\infty, \|b\|_\infty\leq 1$ is the image of $\h^2$ under the operator $(I-T_bT_{\overline{b}})^\frac{1}{2}$ endowed with the inner product $$\langle(I-T_bT_{\bar{b}})^\frac{1}{2}f \ , \ (I-T_bT_{\bar{b}})^\frac{1}{2}g \ \rangle_b\ = \ \langle f, g\rangle_{\h^2} \ , \ f,g\in ker((I-T_bT_{\bar{b}})^\frac{1}{2})^\perp$$ where $\langle \cdot\rangle_{\h^2}$ is the usual inner product on $\h^2$ and $T_b$ denotes the Toeplitz operator with respect to $b$.
	
For a detailed survey of de Branges-Rovnyak spaces, the reader is referred to \cite{ball, vol1, vol2, sar1}. 

     Throughout this paper, we choose a symbol $b$ to be a nonextreme point in the unit ball of $\h^\infty$. The choice of a nonextreme symbol $b$ ensures following important results which are useful in the sequel: 
     \begin{enumerate}
     	\item The polynomials are dense in $\H(b)$ (see \cite[Section 23.4]{vol2}) 
     	\item The multiplication operator $M_z$ on $\H^2$ leaves $\H(b)$ invariant (see \cite[Proposition 4]{kz}).
     	\item There exists a unique outer function $a$ such that $a(0)>0$ and $|a(\zeta)|^2+|b(\zeta)|^2=1 ,\ \mbox{a.e on }\T$. Such a function is called the \textit{pythagorean mate of $b$} and $(b,a)$ are said to form a \textit{pair}. Further, $$f\in \H(b)\iff T_{\overline{b}}f\in T_{\overline{a}}\h^2$$ and $T_{\overline{a}}\h^2$ is dense in $\H(b)$ (see \cite[pg 24-25]{sar1}).
     \end{enumerate} 

	
	\section{Cyclic, Analytic, Completely Hyperexpansive Operators With Finite Rank Defect Operator }\label{secA}
		In this section, we provide a characterization for the measure $\mu$, so that the multiplication operator $M_z$ on the Dirichlet space $D(\mu)$, has defect operator of finite rank. At this stage, we refer the readers to the analogous result \cite[Theorem 6.1]{sc}, proved in the context of a $2$-isometry. It turns out that a matrix decomposition of a completely hyperexpansive operator, as described in \cite{vms}, turns out to be useful in proving the characterization.   
	\begin{proposition}\label{l1}
		Let $\mu$ be a finite positive Borel measure supported on $\overline{\D}$ and let $\Delta_{M_z}=M_z^*M_z-I$ be the defect operator of the multiplication operator on the Dirichlet  space $D(\mu)$. Then the following are equivalent:
		\begin{enumerate}
			\item[(i)] $\Delta_{M_z}$ is of rank $k$
			\item[(ii)] There exist constants $c_i>0 , 1\leq i\leq k$ and $\zeta_i\in\overline{\D}, 1\leq i\leq k$ such that $\mu=\sum_{i=0}^k c_i\delta_{\zeta_i}$
		\end{enumerate} 
	\end{proposition}
	\begin{proof}
	
		\indent Suppose that $\Delta_{M_z}$ is of rank k.\\
		It is known that $M_z$ on $D(\mu)$ is cyclic, analytic and completely hyperexpansive (\cite{ale}). Using the decomposition theorem for completely hyperexpansive operators (\cite[Proposition 3.1]{vms}), we have a decomposition for $M_z$ as follows:\\
		$M_z=\begin{bmatrix}
			V & E\\ 0 & W\\
		\end{bmatrix}$ defined on $\ker(\Delta_{M_z})\oplus\overline{ran(\Delta_{M_z})}$\\
		where \begin{enumerate}
			\item[(a)] $V$ is an isometry
			\item[(b)] $V^*E=0$
			\item[(c)] If $D=E^*E+W^*W-I$, then $D\geq 0$ and $\ker(D)=\{0\}$
			\item[(d)] Letting $M_W(T)=W^*TW$ we have $(I-M_W)^{m-1}(D)\geq 0, \\ 1\leq m\leq k $
		\end{enumerate}
		We note that $\Delta_{M_z}=\begin{bmatrix}
			0 & 0\\ 0 & D\\
		\end{bmatrix}$
		where $D$ is invertible and self- adjoint.\\
	Now, letting $m=2$ in (d) above, we have \begin{equation}\label{E1}
			D\geq W^*DW
		\end{equation}
		Further, as $D$ is invertible, the inequality  (\ref{E1}) yields,
		\begin{eqnarray*}
			(D^{\frac{1}{2}}WD^{-\frac{1}{2}})^*(D^{\frac{1}{2}}WD^{-\frac{1}{2}}) &=& D^{-\frac{1}{2}} W^*DWD^{-\frac{1}{2}} \\
			&\leq & D^{-\frac{1}{2}} D D^{-\frac{1}{2}}\ = \ I
		\end{eqnarray*}
		Therefore , $(D^{\frac{1}{2}}WD^{-\frac{1}{2}})$ is a contraction. Now, let $B=\{v_1,v_2,\cdots v_k\}$ be an orthonormal basis of ${ran(\Delta_{M_z})}$ comprising of eigen vectors of $D^{\frac{1}{2}}WD^{-\frac{1}{2}}$ corresponding to the eigen values $\zeta_i , \ 1\leq i\leq k$,  where each $\zeta_i\in\overline{\D}$.\\
		Now referring to \mbox{(\cite[Proposition (1.6)]{ale})}, we have   \begin{equation}\label{rf1}
			\|zf\|^2_{D(\mu)}-\|f\|^2_{D(\mu)}=\int_{\overline{\D}} |f|^2\ \mathrm{d}\mu \ , \ \forall f\in D(\mu) .  
		\end{equation} By polarization identity, one obtains,
		\begin{equation}\label{rf2}
			\langle zf,zg\rangle- \langle f,g\rangle=\int_{\overline{\D}} f\bar{g}\ \mathrm{d}\mu \ , \ \forall f,g\in D(\mu).
		\end{equation}
		So, for $m, n\geq0$,
		\begin{eqnarray*}
			\langle \Delta_{M_z} M_z^n1, z^m\rangle &=& \langle z^{n+1} , z^{m+1}\rangle-\langle z^n , z^m \rangle\\&=& \int_{\overline{\D}} z^n\bar{z}^m\ \mathrm{d}\mu.
		\end{eqnarray*}
		That is \begin{equation}\label{E3}
			\langle \Delta_{M_z} M_z^n1, M_z^m1\rangle=\int_{\overline{\D}} z^n\bar{z}^m\ \mathrm{d}\mu.
		\end{equation}
		 Now note that, $\{D^{-\frac{1}{2}}v_1, D^{-\frac{1}{2}}v_2, \cdots , D^{-\frac{1}{2}}v_k \}$ also forms a basis for ${ran(\Delta_{M_z})}$.\\ Let $P$ be the projection onto ${ran(\Delta_{M_z})}$. Therefore, $P1=\sum_{i=1}^{k}\beta_i D^{-\frac{1}{2}}v_i$ for some $\beta_i\in\C , 1\leq i\leq k$. 
		Therefore, for $m, n\geq0$,
		\begin{eqnarray*}
			\langle \Delta_{M_z} M_z^n1, z^m\rangle &=&	\langle \Delta_{M_z} z^n, M_z^m1\rangle.\\
				&=& 	\langle DW^nP1, W^mP1\rangle 
		\end{eqnarray*}
	which implies
		\begin{equation}\label{E5}
			\langle \Delta_{M_z} M_z^n1, M_z^m1\rangle=\sum_{i=1}^{k}\zeta_i^n\bar{\zeta_i}^m|\beta_i|^2.
		\end{equation}
		Equations (\ref{E3}) and (\ref{E5}) imply that
		$$ \int_{\overline{\D}} z^n\bar{z}^m\ \mathrm{d}\mu=\sum_{i=1}^{k}\zeta_i^n\bar{\zeta_i}^m|\beta_i|^2.$$
		Since polynomials are dense in $D(\mu)$ (\cite[Corollary 7.3.4]{primer}), by Stone-Weiertrass's theorem, we have $$\mu=\sum_{i=1}^{k}|\beta_i|^2\delta_{\zeta_i}.$$
		\indent 	Conversely, suppose that there exist $c_i>0 $ and $\zeta_i\in\overline{\D} , 1\leq i \leq k$ such that $\mu=\sum_{i=1}^{k}c_i\delta_{\zeta_i}$.\\
		As $\mu$ is supported at $\zeta_i\in\overline{\D}, \  1\leq i\leq k$, for every $f\in D(\mu)$,
		$$\|f\|^2_{D(\mu)}=\|f\|^2_{\h^2}+\int_{\overline{\D}}\left\|\ \frac{f-f(\zeta)}{z-\zeta}\right\|_{\h^2}^2\mathrm{d}\mu(z)=\|f\|^2_{\h^2}+\sum_{i=1}^{k}c_i\left\|\ \frac{f-f(\zeta_i)}{z-\zeta_i}\right\|_{\h^2}^2<\infty.$$
		This implies that $$f\in D(\mu) \iff \left\|\ \frac{f-f(\zeta_i)}{z-\zeta_i}\right\|_{\h^2}^2<\infty , \ 1\leq i \leq k.$$
		Since $D(\mu)$ is a reproducing kernel Hilbert space, point evaluations must be bounded.
		Thus by the reproducing property, there exist linearly independent vectors $\kappa_{\zeta_1} ,\kappa_{\zeta_2}, \cdots \kappa_{\zeta_k}$in  $D(\mu)$ such that $$\langle f,\kappa_{\zeta_i} \rangle=f(\zeta_i) \ , \ \mbox{for every} \ f\in D(\mu), 1\leq i\leq k$$
		We now show that $\Delta_{M_z}$ is of rank $k$.\\
		Let $F=\bigvee \{\kappa_{\zeta_i}| 1\leq i\leq k\}$.\\
		Claim: $\overline{ran(\Delta_{M_z})}=F$ \\
		Now, for $f\in D(\mu)$ consider,
		\begin{eqnarray*}
			\langle \Delta_{M_z}f,f \rangle &=& \|M_zf\|^2_{D(\mu)}-\|f\|^2_{D(\mu)}\\
			&=&\int_{\overline{\D}} |f|^2\ \mathrm{d}\mu\\
			&=& \sum_{i=1}^{k}c_i|f(\zeta_i)|^2\\
			&=& \sum_{i=1}^{k}c_i|\langle f, \kappa_{\zeta_i}\rangle|^2.
		\end{eqnarray*}
		If $f\in F^{\perp}$ then $\langle \Delta_{M_z}f,f \rangle =0$. Since the operator $\Delta_{M_z} \geq 0$,  $\Delta_{M_z}f=0$ i.e $f\in \ker(\Delta_{M_z})$.
		Therefore, $F^\perp\subseteq \ker(\Delta_{M_z})$ or $\overline{ran(\Delta_{M_z})}\subseteq F$. \\ Similarly, if $f\in \ker(\Delta_{M_z})$ then $\Delta_{M_z}f=0$ implies $\sum_{i=1}^{k}c_i|\langle f, \kappa_{\zeta_i}\rangle|^2=0$. That is, $\langle f, \kappa_{\zeta_i}\rangle=0$ for every $1\leq i\leq k\implies f\in F^\perp$. In other words, $\ker(\Delta_{M_z})\subseteq F^\perp$ or $F\subseteq \overline{ran(\Delta_{M_z})}$. Therefore, $\overline{ran(\Delta_{M_z})}=F $ and hence $\Delta_{M_z}$ must be of rank $k$.
	\end{proof}
		
	\section{ Completely Hyperexpansive Shift on a  $\H(b)$ space}\label{secB}
	
	Let $b$ be a non-extreme point of the unit ball of $\h^\infty$ and let $S$ denote the operator of  multiplication by $z$  on $\h^2$. The space $\H(b)$ is invariant under $S$. Let $T$ denote the restriction of $S$ to ${\H(b)}.$ In other words, $T$ is the multiplication by $z$ on the space $\H(b)$.\\ \indent We now obtain the necessary and sufficient conditions on $b$, so that $T$ is a completely hyperexpansive operator.
	\begin{proposition}{\label{T1}}
		Let $b$ be a nonextreme function in the unit ball of $H^{\infty}$. Then the restriction of the shift operator to the de Branges Rovnyak space $H(b)$ is completely hyperexpansive if and only if $b(z)=\frac{c+\gamma z}{1-\beta z}$ with $c,\gamma ,\beta\in\C$ and $|\beta|<1$.
			\end{proposition}
	\begin{proof}
		Suppose that $T$ is a complete hyperexpansion. Then, by definition, $T$ is a $k$-hyperexpansion for every $k\geq 1$. In particular, it is a $2$-hyperexpansion. Using the necessary and sufficient conditions for $T$ to be a $2$-hyperexpansion (Theorem \ref{t1.1}) , $T$ is $2$-hyperexpansive (or concave) if and only if $b$ has the form $b(z)=\frac{c+\gamma z}{1-\beta z}$ where $c,\gamma ,\beta\in\C$ and $|\beta|<1$.\\ \indent 
		Conversely, if \begin{equation*} b(z)=\frac{c+\gamma z}{1-\beta z}\end{equation*} where $c,\gamma ,\beta\in\C$ and $|\beta|<1$, then its pythagorean mate, $a(z)=\frac{\rho-\sigma z}{1-\beta z}$ where $\rho$ and $\sigma$ follow the conditions \begin{enumerate}
			\item[(i)] $\rho\geq |\sigma|$
			\item[(ii)] $\rho^2+|\sigma|^2=1+|\beta|^2-|c|^2-|\gamma|^2\ \mbox{and }\ \rho^2|\sigma|^2=|\beta+\overline{c}\gamma|^2$
			\item[(iii)] $\mbox{arg}(\rho)=\mbox{arg}(\beta+\overline{c}\gamma)$.
		\end{enumerate}
		
		We observe that for any $h\in\H(b)$,
		\begin{eqnarray*}
			\langle (I-3T^*T+3{T^*}^2T^2-{T^*}^3T^3)h , h\rangle_b &=&(1+\|b\|_b^2)\left(1-\left|\beta-\frac{a'(0)}{a(0)}\right|^2\right)\\ 
			& & \left\{|\langle h,{S^*b}^2\rangle_b|^2-2\mathrm{Re}\left(\bar{\frac{{a'(0)}}{a(0)}}\langle h,S^*b\rangle_b\overline{\langle h,{S^*}^2b\rangle_b}\right)\right.\\
			& &\left.+\left[\left|\frac{a'(0)}{a(0)}\right|^2-1\right]|\langle h,S^*b\rangle_b|^2\right\}\\
			&=&\left(1-\left|\beta-\frac{a'(0)}{a(0)}\right|^2\right)\langle (I-2T^*T+{T^*}^2T^2)h , h \rangle_b
		\end{eqnarray*}
		Inductively, it follows that, for every $n\geq 2$ and every $h\in\H(b)$,
		$$\left\langle\left(\sum_{i=0}^{n}(-1)^i{n \choose i} {T^*}^iT^{i}\right)h , h\right\rangle_b=\left(1-\left|\beta-\frac{a'(0)}{a(0)}\right|^2\right)^{n-2}\langle (I-2T^*T+{T^*}^2T^2)h , h \rangle_b$$
		Therefore, $T$ is completely hyperexpansive if $$\langle (I-2T^*T+{T^*}^2T^2)h , h \rangle_b\leq0 \ \mbox{or} \ 1-\left|\beta-\frac{a'(0)}{a(0)}\right|^2= 0.$$
		In either case, by appealing to Theorem \ref{t1.1}, we conclude that $T$ must be $2$-hyperexpansive and thus it follows that $T$ is completely hyperexpansive.
		
	\end{proof}
	\noindent We conclude this section by proving a result required in the proof of the main theorem of this paper.
	\begin{proposition}{\label{p2}}
		Let $b$ be a non-extreme point of unit ball of $\h^\infty$ and $a$ be its pythagorean mate. Then the defect operator $\Delta_{T}=T^*T-I$ is of rank 1 and its range is spanned by $S^*b$ and its non-zero eigen value equals $|a(0)|^{-2}\|S^*b\|_b^2$.
	\end{proposition}
	\begin{proof}
		Since $(b,a)$ forms a pair, we have
		\begin{eqnarray*}
			T^*T&=&I+|a(0)|^{-2} S^*b\otimes S^*b\\
			\therefore \Delta_{T}=T^*T-I &=& |a(0)|^{-2} S^*b\otimes S^*b
		\end{eqnarray*}
	where  $(x\otimes y)$ denotes the operator $ (x\otimes y)(z)=\langle z,y\rangle_{b} x \ ,  \ z\in \H(b).$\\
		Hence the matrix of the operator $\Delta_{T}$ relative to the decomposition $\H(b)=\C S^*b\oplus(\C S^*b)^\perp$ is given by $$\begin{bmatrix}
			|a(0)|^{-2}\|S^*b\|_b^2 & 0\\
			0 & 0
		\end{bmatrix}$$
		which implies the required result.
	\end{proof}
	
	\section{Proof of Theorem \ref{t1}}\label{secC}
	This section is devoted to the proof of the main result of the paper viz Theorem \ref{t1}.\\
	\indent Let $b\in \h^\infty$ be nonextreme with $\|b\|_\infty\leq 1$ and let $\mu$ be a finite positive Borel measure supported on $\overline{\D}$. Let $T=M_z|_{\H(b)}$ denote the restriction of multiplication by $z$ operator to $\H(b)$ space.
\\\vskip0.1cm \noindent	\textbf{Proof of Theorem \ref{t1}:}
	\begin{proof}

		Suppose that $\H(b)=D(\mu)$ with equality of norms. 
		\\\indent As $M_z$  on $D(\mu)$ is cyclic, analytic and completely hyperexpansive, $T$ must be cyclic, analytic and completely hyperexpansive. Therefore, Theorem \ref{T1} implies that $b$ must be of the form \begin{equation*} b(z)=\frac{c+\gamma z}{1-\beta z}
		\end{equation*} where $c,\gamma ,\beta\in\C$ and $|\beta|<1$. \\ Similarly, Proposition \ref{p2} implies that the corresponding defect operator $\Delta_{T}$ has rank 1. Therefore, the corresponding defect operator $\Delta_{M_z}$ on $D(\mu)$ must also have rank 1. So, using Proposition \ref{l1} we conclude that $$\mu=|\alpha|^2\delta_\lambda\ , \ \mbox{for some} \ \alpha\in\C , \lambda\in\overline{\D}.$$ 
		Now, for Cauchy kernels $k_w(z)=\frac{1}{1-\overline{w}z}$ we have,\\
		$$D_\lambda(k_w)=\frac{1}{2\pi}\int_\T \left|\frac{k_w(z)-k_w(\lambda)}{z-\lambda}\right||\mathrm{d}z|=\left\| \frac{\overline{w}}{1-\overline{w}\lambda}k_w\right\|_{\h^2}^2$$
		Since,
		$$D_\mu(f)=\int_{\overline{\D}}D_\zeta(f)\mathrm{d}\mu(\zeta)$$ therefore we have,
		$$D_\mu(k_w)=|\alpha|^2\left\| \frac{\overline{w}}{1-\overline{w}\lambda}k_w\right\|_{\h^2}^2=\frac{|\alpha|^2|w|^2}{|1-\overline{\lambda}w|^2(1-|w|^2)}$$
		Similarly, $$\|k_w\|_b^2=\|k_w\|_{\h^2}^2+\|k_w^+\|_{\h^2}^2=\|k_w\|_{\h^2}^2+\frac{1}{(1-|w|^2)}\frac{|b(w)|^2}{|a(w)|^2}$$ 
		But as $\H(b)=D(\mu)$ with equality of norms we have,  \begin{equation}\label{E9} \frac{1}{(1-|w|^2)}\frac{|b(w)|^2}{|a(w)|^2}=\|k_w^+\|_{H^2}^2=D_\mu(k_w)=\frac{|\alpha|^2|w|^2}{|1-\overline{\lambda}w|^2(1-|w|^2)}\end{equation} \\ In particular, letting $w=0$ we  obtain $b(0)=0$. That is, $b(z)=\frac{\gamma z}{1-\beta z}$.\\
		
		\indent	Further , if $\phi=\frac{b}{a}$ , then equation (\ref{E9}) gives us $$|\phi(w)|^2=\frac{|\alpha|^2|w|^2}{|1-\overline{\lambda}w|^2}$$ As the result is true for every $w$ , we obtain $$\phi(z)=\frac{\alpha z}{1-\overline{\lambda}z}.$$
		
		Let $p(z)=\alpha z , \ q(z)=1-\overline{\lambda}z$. 
		We observe that $\phi$ is a Smirnov class function. Then, the pair $(b,a)$ can be obtained by the formulation, 
		$$|a|^2=\frac{|q|^2}{|p|^2+|q|^2} \mbox{ and}\ |b|^2=\frac{|p|^2}{|p|^2+|q|^2} \  \mbox{a.e on}\  \T\ \mbox{(see \cite[Lemma 23.27]{vol2})}.$$\\ Therefore,  
		\begin{eqnarray*}
			\frac{|\gamma|^2}{|1-\beta z|^2}&=&\frac{|\alpha|^2}{|1-\overline{\lambda}z |^2+|\alpha|^2|z|^2} \ , \ \mathrm{a.e \ on} \ \T
		\end{eqnarray*}
		\begin{eqnarray*}
			\therefore \ |\gamma|^2\left((1+|\lambda|^2+|\alpha|^2)+2\re(\overline{\lambda}z)\right)&=&|\alpha|^2\left((1+|\beta|^2)+2\re(\beta z)\right) ,\ \mathrm{a.e \ on} \ \T
		\end{eqnarray*}
		Substituting $z=1,\ -1, \ i,\ -i$ we obtain the following: 
		\\If $\lambda=0$ then
		$$|\gamma|^2=\frac{1}{1+|\alpha|^2}$$ and therefore $\beta=0$
		\\Similarly, if $\lambda\neq 0$, then
		$$|\gamma|^2=\frac{|\alpha|^2}{2|\lambda|^2}\left((1+|\alpha|^2+|\lambda|^2)\pm\sqrt{(1+|\alpha|^2+|\lambda|^2)^2-4|\lambda|^2}\right).$$ 
		But, $$|\gamma|^2=\frac{|\alpha|^2}{2|\lambda|^2}\left((1+|\alpha|^2+|\lambda|^2)+\sqrt{(1+|\alpha|^2+|\lambda|^2)^2-4|\lambda|^2}\right)$$ implies $|\beta|>1$.\\
		Therefore,
		$$|\gamma|^2=\frac{|\alpha|^2}{2|\lambda|^2}\left((1+|\alpha|^2+|\lambda|^2)-\sqrt{(1+|\alpha|^2+|\lambda|^2)^2-4|\lambda|^2}\right).$$ Hence $$\beta=\frac{\overline{\lambda}}{2|\lambda|^2}\left((1+|\alpha|^2+|\lambda|^2)-\sqrt{(1+|\alpha|^2+|\lambda|^2)^2-4|\lambda|^2}\right).$$
		
		\noindent Letting $A_\lambda=\gamma $ and $B_\lambda=\beta$ we obtain the required result.\\
		\indent Conversely, suppose that  $\mu=|\alpha|^2\delta_\lambda ,  \ \mbox{for some }\alpha\in\C ,\ \lambda\in\overline{\D}$ and $b(z)=\frac{A_\lambda z}{1-B_\lambda z}$ where $ A_\lambda$ and $B_\lambda$ are constants depending on $\alpha$ and $\lambda$ such that 
		$$|A_\lambda|^2=\begin{cases}
			\left(\frac{|\alpha|^2}{2|\lambda|^2}\left((1+|\alpha|^2+|\lambda|^2)-\sqrt{(1+|\alpha|^2+|\lambda|^2)^2-4|\lambda|^2}\right)\right) & \mbox{, if} \ \lambda\neq 0\\
			\left(\dfrac{1}{1+|\alpha|^2}\right) & \mbox{, if} \ \lambda= 0
		\end{cases}$$
		and $$ B_\lambda= \frac{|A_\lambda|^2\overline{\lambda}}{|\alpha|^2}.$$	 
		Then by Proposition \ref{l1}, the defect operator of $M_z$ on $D(\mu)$, $\Delta_{M_z}$ must be of rank 1.\\  Similarly, applying Proposition \ref{T1} and Proposition \ref{p2}, we conclude that the operator $T$ must be cyclic, analytic, completely hyperexpansive and its defect operator, $\Delta_{T}$ must be of rank 1. 
		
		Recall that, by Aleman's result \cite[Theorem 2.5]{ale}, T must be unitarily equivalent to the multiplication by $z$ operator on $D(\mu')$ for some non-negative, Borel measure $\mu'$ supported on $\overline{\D}$ and $\H(b)=D(\mu')$ with equality of norms. But, previous part of this theorem implies that $\mu$ and $\mu'$ must be equal.
	\end{proof}
	The above result provides a formulation for computing the non-extreme symbol $b$, given a positive, point mass measure $\mu$ on $\overline{\D}$, so that $D(\mu)=\H(b)$ with equality of norms. We use the following examples to illustrate the construction.
	\begin{example}
		Suppose $\mu=\delta_0$ i.e $|\alpha|^2=1$ and $\lambda=0$. Then by the formula prescribed in the above theorem, one obtains 
		$A_\lambda=\frac{1}{\sqrt{2}}$ and $B_\lambda=0$. 
		
	\end{example}
	\begin{example}
		Suppose $\mu=\delta_{\frac{1}{2}}$ i.e $|\alpha|^2=1$ and $\lambda=\frac{1}{2}$. Then,\\ 
		$A_\lambda=\left(\dfrac{{9-\sqrt{65}}}{2}\right)^\frac{1}{2}$ and $B_\lambda=\dfrac{(9-\sqrt{65})}{4}$. 
		
	\end{example}
	\noindent We conclude this paper by deducing Corollary \ref{c1} from Theorem \ref{t1}.\\\vskip0.25cm
\noindent\textbf{Proof of Corollary \ref{c1}:}
	\begin{proof}
		Letting $|\lambda|=1$ in Theorem \ref{t1}, we obtain
		$$|A_\lambda|^2=\frac{|\alpha|^2}{2}\left(2+|\alpha|^2-|\alpha|\sqrt{(4+|\alpha|^2)}\right),$$ Similarly, $ B_\lambda= \frac{|A_\lambda|^2\overline{\lambda}}{|\alpha|^2}$	 implies $$ |B_\lambda|=\frac{|A_\lambda|^2}{|\alpha|^2}\neq0$$
		\\Now , 
		\begin{eqnarray*}
			&	|B_\lambda|+|A_\lambda|=1 \\
			\iff & \frac{|A_\lambda|^2}{|\alpha|^2}+|A_\lambda|=1\\
			\iff& |A_\lambda|^2=\frac{|\alpha|^2}{2}\left(2+|\alpha|^2\pm|\alpha|^2\sqrt{(4+|\alpha|^2)}\right)
		\end{eqnarray*} 
		Therefore,
		$$|\alpha|^2=\frac{|A_\lambda|^2}{|B_\lambda|}, \ \lambda=\frac{\overline{B_\lambda}}{|B_\lambda|}.$$
	This completes the proof.
	\end{proof}
	\subsection*{Acknowledgments:} 
	\begin{enumerate}
		\item The authors would like to thank Prof. Sameer Chavan (IIT, Kanpur) for several useful discussions throughout the preparation of this paper.
		\item The present work is carried out at the research centre at the Department of Mathematics, S. P. College, Pune, India (autonomous).
	\end{enumerate}
	
	
	\bibliographystyle{plain}	\bibliography{ACH_dBR5.bib}

\end{document}